\documentclass[11pt,reqno,a4paper]{amsart}

\usepackage{amsthm,amsfonts,amsmath,amssymb}
\usepackage[mathscr]{eucal}

\usepackage[pdftex,bookmarks=true]{hyperref}
\usepackage{enumerate}

\newcommand{\zt}{\zeta}

\newcommand{\bdy}{\partial}
\newcommand{\OM}{\Omega}
\newcommand{\disk}{\mathbb{D}}

\newcommand{\smoo}{\mathcal{C}}

\newcommand{\bcdot}{\boldsymbol{\cdot}}
\newcommand\cjac{{\sf Jac}_{\C}}
\newcommand{\poin}{{\sf p}_{\!{\raisebox{-3pt}{$\scriptstyle{\mathbb{D}}$}}}}

\newcommand{\C}{\mathbb{C}}
\newcommand{\R}{\mathbb{R}}
\newcommand{\N}{\mathbb{N}}
\newcommand{\Z}{\mathbb{Z}}

\theoremstyle{plain}
\newtheorem{thm}{Theorem}[section]

\newtheorem{result}[thm]{Result}
\newtheorem{lemma}[thm]{Lemma}
\newtheorem{prop}[thm]{Proposition}

\theoremstyle{definition}
\newtheorem{defn}[thm]{Definition}
\newtheorem{example}[thm]{Example}
\theoremstyle{remark}
\newtheorem{rmk}[thm]{Remark}

\newenvironment{enumeratei}{\begin{enumerate}%
  [$(i)$]}%
  {\end{enumerate}}

\begin{document}

\title[Proper holomorphic mappings of balanced domains in $\C^n$]{Proper
holomorphic mappings of \\ balanced domains in $\C^n$}
\author{Jaikrishnan Janardhanan}
\address{Department of Mathematics, Indian Institute of Science, Bangalore 560012, India}
\email{jaikrishnan@math.iisc.ernet.in}
\thanks{This work is supported by a UGC Centre for Advanced Study grant
and by a scholarship from the IISc.}

\begin{abstract} 
  We extend a well-known result, about the unit ball, by H. Alexander to a class of
  balanced domains in $\C^n, \ n > 1$. Specifically: we prove that any proper holomorphic
  self-map of a certain type of balanced, finite-type domain in $\C^n, 
  \ n > 1$, is an automorphism. The main novelty of our
  proof is the use of a recent result of Opshtein on the behaviour of the iterates of
  holomorphic self-maps of a certain class of domains. We use Opshtein's theorem, together
  with the tools made available by finiteness of type, to deduce that the aforementioned
  map is unbranched. The monodromy theorem then delivers the result. 
\end{abstract}
\maketitle

\section{Introduction}
The central result of this paper is:

\begin{thm}\label{THM:BAL}
  Let $\OM \subset \C^n, \ n > 1$, be a smoothly bounded balanced
  domain of (D'Angelo) finite type. Assume that $\OM$ has a smooth defining 
  function that is plurisubharmonic in $\OM$. Then every proper holomorphic self-map 
  $F:\OM \to \OM$ is an automorphism.
\end{thm}

We recall that a domain $\OM$ is said to be balanced if whenever $z \in \OM$ and 
$\lambda \in \C, |\lambda| \leq 1$, then we have $\lambda z \in \OM$. The above theorem
is motivated by the famous result of H. Alexander \cite{alexander1977proper}
showing the non-existence of non-injective proper holomorphic self-maps of the unit ball
in $\C^n, \ n >1$ --- which the above theorem generalizes. Since its appearance,
Alexander's theorem has been extended in many different ways. We refer the interested reader
to Section~3 of the survey article by Forstneri{\v{c}} \cite{forstneric1992survey}. One 
such result is that of Bedford and Bell 
\cite{bedford1982proper}, which extends Alexander's theorem to bounded pseudoconvex 
domains having real-analytic boundaries. Recent work relating to Alexander's theorem has,
therefore, focused on domains whose boundaries need not be real-analytic. See, for
instance, the work of Berteloot \cite{berteloot1998holomorphic} and Coupet, Pan and Sukhov
\cite{coupet1999proper,coupet2001proper}. 
\smallskip

An alternative strategy for deriving a conclusion similar to that of Alexander's
theorem, especially when a domain does \emph{not} have real-analytic boundary, is to
impose a condition on its automorphism group. With a hypothesis on the automorphism group,
there are often interesting interactions between the 
symmetries of the domain and the boundary geometry. These interactions allow one to 
simplify the structure of the branch locus of a branched proper holomorphic self-map. 
\smallskip

However, there is a new ingredient to our proof, which we briefly introduce here. In a
relatively recent paper, Opshtein proved the following result:

\begin{result}[\cite{opshtein2006dynamics}, Th{\'e}or{\`e}me A]\label{RES:OPHI}
   Let $D \subset \C^n, \ n \geq 2$, be a smoothly bounded domain having a smooth 
   defining function that is plurisubharmonic in $D$. Let $f : D \to D$ be a proper holomorphic 
   self-map that is recurrent. Then the limit manifold of $f$ is necessarily of dimension
    higher than $1$. 
\end{result}
\noindent Refer to Section~\ref{SEC:ITER} for the definitions of the terms that occur in 
the above result. In \cite{opshtein2006dynamics}, Opshtein suggests that his
results could serve as a new set of tools for establishing Alexander-type theorems. We 
found Opshtein's viewpoint very useful in the context of Theorem~\ref{THM:BAL}: 
Result~\ref{RES:OPHI} plays a key role in our proof of Theorem~\ref{THM:BAL} (which we 
shall presently explain).
\smallskip

Another influence behind this work is the following result:

\begin{result}[Coupet, Pan and Sukhov \cite{coupet1999proper}]\label{RES:CPS}
  Let $\OM \subset \C^2$ be a smoothly bounded balanced pseudoconvex 
  domain of (D'Angelo) finite type. Then every proper holomorphic self-map 
  $F:\OM \to \OM$ is an automorphism.
\end{result} 

\begin{rmk}
  In a later paper \cite{coupet2001proper}, Coupet, Pan and Sukhov have extended their 
  result to smoothly bounded pseudoconvex quasi-circular domains of finite type in $\C^2$.
\end{rmk}

Let $\OM$ and $F$ be as in Theorem~\ref{THM:BAL}. Our proof of Theorem~\ref{THM:BAL} and 
the proof of Result~\ref{RES:CPS} share some key ideas. As $\OM$ is simply connected, to 
prove that $F$ is an automorphism it suffices --- in view of the monodromy theorem --- to
establish that $F$ is unbranched. Thus, a natural approach to proving Theorem~\ref{THM:BAL} is
to \emph{assume} that $F$ is branched, and use this assumption together with
the properties of $\OM$ to reach a contradiction. To this end, we prove
the following result that establishes an important property of the branch locus
of $F$.

\begin{prop}\label{PROP:LOCUS}
  Let $\OM \subset \C^n, \ n > 1,$ be a smoothly bounded balanced pseudoconvex 
	domain of (D'Angelo) finite type. Let $F : \OM \to \OM$ be a proper holomorphic
	mapping, and
	assume that the branch locus $V_F := \{z \in \OM : \cjac(F)(z) = 0\} \neq
	\emptyset$. Let	$X$ be an irreducible component of $V_F$. Then for each $z \in X$,
	the set $(\C \bcdot z) \cap \OM$ is contained in $X$.
\end{prop}

We ought to point out that Proposition~\ref{PROP:LOCUS} is a hypothetical
statement. \emph{If} $F$ as in Theorem~\ref{THM:BAL} were branched, \emph{then} it
would have the above structure. The thrust of our proof is that $F$ can never be branched.
\smallskip

Coupet, Pan and Sukhov have proved a version of Proposition~\ref{PROP:LOCUS} 
for domains in $\C^2$ in which the domain need not be balanced, but is only 
required 
to admit a transverse $\mathbb{T}$-action. The point that is worth highlighting here is  
that by restricting ourselves to balanced domains, we are able to
give an almost
entirely elementary proof of Proposition~\ref{PROP:LOCUS}, and that these methods have 
one significant payoff: we \emph{do not} have to assume in the above results that the 
$\mathbb{T}$-action is transverse. We bypass the 
need for transversality by using some results from dimension theory. These results are
insensitive  to dimension, which allows us to state and prove Theorem~\ref{THM:BAL} in
$\C^n$ for all $n > 1$.
\smallskip

A key object used in the proofs of Proposition \ref{PROP:LOCUS} and Theorem~\ref{THM:BAL}
is
the function $\tau : \bdy \OM \to \Z_+ \cup \{ 0 \}$ introduced by Bedford and Bell (see
\cite{bedford1982proper, bell1982local}). The number $\tau(p)$ is the order of vanishing 
in the tangential directions 
of the Levi determinant of a smoothly bounded pseudoconvex domain at the point $p \in
\bdy \OM$. The function $\tau$ has been used to study the branching behaviour
of proper holomorphic mappings in many earlier results. It turns out that, owing to the
hypothesis that $\OM$ is of (D'Angelo) finite type, $\tau$ is bounded on $\bdy \OM$.
We refer the reader to Section~\ref{SEC:GEOM} for a more precise discussion.
\smallskip

Using Proposition~\ref{PROP:LOCUS}, one can prove that $F^{-1}\{0\} = \{0\}$. 
It is at this point that our methods and the methods of Coupet--Pan--Sukhov
diverge. We make use of Opshtein's theorem to deal with the complexities posed by 
dimension.
The key steps of Theorem~\ref{THM:BAL} may be summarized as follows:
\begin{itemize}
 \item We begin by assuming that $F$ is branched.
 \item If $z_0$ is a point in $F^{-1}\{0\}$ that is not $0$, it
 follows from Proposition~\ref{PROP:LOCUS} that 
 $(\C\bcdot z_0)\cap \overline{\OM}$ is contained in a sequence of distinct
 irreducible components of the branch loci of the iterates $F^k$, $k = 1,2,3,\dots$
 We note that, by Lemma~\ref{LEM:EXTBELL}, $F$ extends holomorphically to a
 neighbourhood of $\overline{\OM}$, whence these irreducible components extend through
 $\bdy \OM$.
 \item We pick a point $q\in (\C\bcdot z_0)\cap \bdy\OM$. Each iterate
 of $F$ must be branched at $q$. Using this fact, and that $q$ is located
 on the closures of \emph{distinct} irreducible components of the branch loci of $F^k$, 
 one can show --- using a result of Bell \cite{bell1982local} --- that $\tau$
 must be unbounded on $\bdy \OM$. This is impossible, whence $F^{-1}\{0\} = \{0\}$.
 \item In particular, $F$ fixes $0$, whence there exists a limit manifold, call it
 $M$, associated to the iterates of $F$. The circular symmetry of $\OM$ makes it
 possible to deduce that $M$ is the intersection of a linear subspace of
 $\C^n$ with $\OM$, and that one may assume, without loss of generality,
 that $F|_M$ is given by
 \[
  (z_1,\dots, z_m, 0,\dots, 0)\,\mapsto\,(e^{i\theta_1}z_1,\dots, 
  e^{i\theta_m}z_m, 0,\dots, 0).
 \]
 \item If $m > 1$ then we can find a point $p\in \overline{M}\cap \bdy\OM$
 that \emph{also} lies in the (prolongation of) the branch locus of $F$.
 We examine the orbit of $p$ under the action of the group generated by
 $F|_M$. The behaviour of the function $\tau$ along this orbit contradicts
 the upper semi-continuity of $\tau$. Hence, $m\leq 1$, which, however,
 contradicts Opshtein's theorem: Result~\ref{RES:OPHI} above.
 \item This proves that our assumption that $F$ is branched must be false.
\end{itemize}  
 
As is evident from the above outline, we will need some definitions and
facts from the theory of (iterative) dynamics of holomorphic self-maps. These
will be presented in Section~\ref{SEC:ITER}. Section~\ref{SEC:BALESS} is 
devoted to stating
and proving certain propositions that are essential to our proofs. The proofs of
Proposition~\ref{PROP:LOCUS} and Theorem~\ref{THM:BAL} will be presented in 
Section~\ref{SEC:BALPROOF}.
\smallskip

Before, we proceed further, we clarify that, in this paper, whenever we use use the word
``smooth'', it will refer to $\smoo^{\infty}$-smoothness unless specified otherwise. By a
smoothly bounded domain, we shall mean a bounded domain whose boundary is 
$\smoo^\infty$-smooth.
\smallskip

\section{Boundary geometry}\label{SEC:GEOM}
In this section, we shall summarize the properties of the function $\tau$ alluded to
above. For the sake
of completeness, we first give the definition of D'Angelo finite type. 

\begin{defn}
  Let $D \subset \C$ be a domain, and let $f:D \to \C$ be a smooth function. We define 
  the \emph{multiplicity} of $f$ at $p \in D$ to be the least positive integer $k$ such
  that the homogeneous polynomial of degree $k$ in the Taylor series of $(f-f(p))$
  around $p$ is not identically zero, defining it to be $+\infty$ if no such $k$ exists.
   The multiplicity of a $\C^n$-valued function 
  is defined to be the minimum of the multiplicities of its components. We denote
  the multiplicity of a function $f$ at a point $p$ by $v_p(f)$.
\end{defn}
\begin{defn}[D'Angelo]
  Let $M \subset \C^n$ be a smooth real hypersurface, and let $p \in M$. Let $r$ be a 
  defining
  function for $M$ in some neighbourhood of the point $p$. We say that $p$ is a point of 
  \emph{finite type} (also known as finite 1-type) if there is a constant $C > 0$ such
  that
  \[
	\frac{v_0(r \circ \phi)}{v_0(\phi)} \leq C,
  \]
  whenever $\phi : \disk \to \C^n$ is a non-constant analytic disk such that $\phi(0) =
  p$.
\end{defn}

It is simple to prove that the
above definition is independent of the choice of the defining function $r$. Whenever we say below
that a bounded domain is of finite type, we shall mean that each boundary point is of finite type in the
sense of the above definition.
\smallskip

We now define a function $\tau$. This
function was introduced by Bedford and Bell \cite{bedford1982proper}, and has been used 
in many results to examine the branching behaviour of proper holomorphic mappings. It 
is used multiple times in our proof of Theorem~\ref{THM:BAL}

\begin{defn}
  Let $D \subset \C^n$ be a smoothly bounded pseudoconvex domain and $r$ a smooth
  defining function for $D$. Define
   \[
	\Lambda_r := \det 
	\begin{bmatrix} 
	  0  & r_{z_i}  \\
	 r_{\overline{z}_j} & r_{z_i \overline{z}_j}
   \end{bmatrix}_{i, j =1}^n,
   \]
   the \emph{determinant of the Levi-form of $r$} (for a justification of this 
   terminology, see \cite[Section~2]{kohn1999sub}). For $p \in \bdy D$, we define
   $\tau(p)$ to be the smallest non-negative integer $m$ such that there is a 
   tangential differential operator $T$ of order $m$ on $\bdy D$ such that
   $T \Lambda_r(p) \neq 0$.
\end{defn}

\begin{rmk}
 As any other defining function $r'$ can be written as $h \cdot r$, where $h$ is a
 positive smooth function defined on some neighbourhood $U$ of the boundary point $p$,
 we see that the number $\tau(p)$ is independent of the choice of $r$. Note that by the
 pseudoconvexity of $D$, $\Lambda_r(p) \geq 0 \ \forall p \in \bdy D$, and $\Lambda_r(p)
 = 0$ if and only if $p$ is a point of weak pseudoconvexity.
\end{rmk}

Observe that $\tau$ is an upper semi-continuous function on $\bdy D$.
\smallskip

Suppose $f: D_1 \to D_2$ is a proper holomorphic mapping between bounded pseudoconvex 
domains with $\smoo^{\infty}$-smooth boundaries that extends smoothly to a $\bdy D_1$-open
neighbourhood of a point $p \in \bdy D_1$.
Let $\rho$ be a defining function for $D_2$ such that $\rho \circ f$ is a
local defining function for $\bdy D_1$ near $p$ (see \cite[Remark 2]
{bedford1984proper}). It follows that
\[
  \Lambda_{\rho \circ f}(z) = |\cjac(f)(z)|^2\Lambda_{\rho}(f(z)),
\]
for $z \in \bdy D_1$ and close to $p$, from which the next result is straightforward to prove.
\smallskip

Let us establish the following notation that we shall use in the remainder of this 
paper. Given any open set $D \subset \C^n$ and a holomorphic map 
$f : D \to \C^n$, $V_f$ will be defined as:
\[
  V_f := \{z \in D : \cjac(f)(z) = 0\}.
\]

\begin{result}[Bell \cite{bell1982local}]\label{RES:TAU}
  Given a proper holomorphic mapping $g : D_1 \to D_2$ between bounded pseudoconvex
  domains in $\C^n, \ n  > 1$, with smooth boundaries, if $g$ extends smoothly to 
  $\bdy D_1$ in a
  neighbourhood of $p \in \bdy D_1$, then $\tau(p) \geq \tau(g(p))$, and when $\tau(p)
  \neq \infty$, the following are equivalent :
  \begin{enumeratei}
	\item 	$\tau(p) = \tau(g(p))$;
	\item     $g$ extends to a local diffeomorphism at $p$;
	\item     $p \not \in \overline{V}_g$.	 
  \end{enumeratei}
\end{result}

The next two results give an idea why the $\tau$ function is relevant to
our main theorem.
\smallskip

Let $D \subset \C^n$ be a smoothly bounded pseudoconvex domain of finite type
and let $p\in \bdy D$. It can be argued that $\tau(p)$ is finite.
As the D'Angelo 1-type of $p$ is finite,
if we pick a non-zero vector $V\in T_p(\bdy D)\cap iT_p(\bdy D)$, then by
definition there exists a $k\in \Z_+$ such that the homogeneous polynomial of degree $k$ in the Taylor
expansion of $\zt\mapsto r(p+\zt V)$ around $0\in \C$ is {\em not} identically zero
($r$ here is a defining function of $D$).
However, as this $k$ depends on the choice of $V\in T_p(\bdy D)\cap iT_p(\bdy D)$, it becomes
quite technical to produce  a {\em single} finite-order tangential differential operator
$T$ at $p$ such that $T\Lambda_r(p)\neq 0$. We could not find an elementary proof
of Result~\ref{RES:FIN} (see below) in the literature, although it has been made use of a number
of times; see, for instance, \cite{pan1991reinhardt}. A recent work of Nicoara
\cite[Main Theorem 1.1]{nicora2012eff} provides an effective upper bound for
$\tau$ in terms of the D'Angelo 1-type. For our purposes, the following consequence of Nicoara's
result suffices:   

\begin{result}\label{RES:FIN}
  Let $D \subset \C^n$ be a smoothly bounded pseudoconvex domain that is of finite type
  in the sense of D'Angelo. Then there is
  an $m\in \Z_+$ such that $\tau(p) \leq m \ \forall p \in \bdy D$.  
\end{result}

The next result is by Coupet, Pan and Sukhov \cite{coupet2001proper}. Since the statement 
of the result below is slightly different from that of  \cite[Lemma~1]{coupet2001proper},
and its proof is sufficiently important for our purposes, we provide a proof of it. Here,
and elsewhere in this paper, $F^\nu$ will denote the $\nu$-th iterate of $F$.
 
\begin{lemma}[a paraphrasing of Lemma~1 of \cite{coupet2001proper}]\label{LEM:VAR}
  Let $\OM \subset \C^n, n > 1$, be a bounded pseudoconvex domain with
  $\smoo^{\infty}$-smooth
  boundary that is also of finite type. Let $F : \OM \to \OM$ be a proper holomorphic 
  mapping (which extends smoothly to $\bdy \OM$ by \cite{catlin1987sub}
  together with \cite{bell1982boundary} or \cite{diederich1982boundary}). Assume that
  $V_F \neq \emptyset$ and let $L_1$ be any
  irreducible component of $V_F$. Inductively select any irreducible component of
  $F^{-1}(L_{\nu})$, $\nu = 1, 2, 3,\dots$, and denote it by $L_{\nu+1}$. Then:
  \begin{enumerate}[$a)$]
   \item $L_{\nu}$ is an irreducible component of $V_{F^{\nu}}$.
   \item $L_i\neq L_j$ if $i\neq j$.
  \end{enumerate}
\end{lemma}

\begin{proof}
  It is elementary to see that  $V_{F^{\nu+1}} \supset F^{-1}(V_{F^{\nu}})$. From this $(a)$
  is immediate. Note that the 
  restriction of $F$ to each $L_{\nu+1}$ is proper, and consequently $F(L_{\nu+1}) =
  L_{\nu}$. We now show that the stated procedure for constructing $\{L_\nu\}$,
  \emph{irrespective of the choices made}, ensures
  that $L_{\nu+1} \neq L_j \ \forall j\leq \nu, \nu=1,2,3,\ldots$. Suppose
  not, and let $m$ be the smallest positive integer such that $L_m = L_{m+p}$ for some
  positive integer $p$. If $m > 1$, then $F(L_m) = F(L_{m+p})$, and so $L_{m-1} =
  L_{m+p-1}$, contradicting the definition of $m$. So, $m = 1$ and $L_1 = L_{1+p}$. Since
  $F^p(L_{1+p}) = L_1$, we have $F^p(L_1) = L_1$. As $L_1 \subset V_{F^p}$, for any $q \in
  \overline{L}_1 \cap \bdy \OM$, we see, by Result~\ref{RES:TAU}, that
  \begin{equation}\label{EQ:TYPE_DES}
	0 \leq \cdots \tau(F^{kp}(q)) < \tau(F^{(k-1)p}(q)) < \cdots < \tau(q).
  \end{equation}
  This, in view of Result~\ref{RES:FIN}, contradicts the finite-type hypothesis on
  $\bdy \OM$.
\end{proof}
The extendability of $F$ as stated in the previous lemma --- and, indeed, a stronger
conclusion --- is guaranteed when $\OM$ is as in Theorem~\ref{THM:BAL}. We shall need this
stronger conclusion, presented in the next lemma, in our proof. This lemma is an easy
consequence of \cite[Theorem~2]{bell1982proper}; see \cite[Lemma 4.2]{bharali2014proper}
for a proof.

\begin{lemma}\label{LEM:EXTBELL}
 Let $f: D_1\to D_2$ be a proper holomorphic map between bounded balanced
 domains. Assume that the intersection of every complex line
 passing through $0$ with $\bdy{D}_1$ is a circle. Then $f$ extends holomorphically to a
 neighbourhood of $\overline{D}_1$.
\end{lemma}

\section{Dynamics of holomorphic mappings}\label{SEC:ITER}

In this section we summarize some material from the theory of (iterative) dynamics of
holomorphic self-maps of a taut manifold. Given complex manifolds $X$ 
and $Y$, ${\rm Hol}(X,Y)$ will 
denote the space of holomorphic mappings from $X$ to $Y$, where the topology on  
${\rm Hol}(X,Y)$ is the compact-open topology. We are interested in the set
$\Gamma(f)$ which is defined to be the set of all limit points of the iterates of a
holomorphic mapping $f \in {\rm Hol}(X,X)$, where $X$ is a taut complex manifold. 
Of course, $\Gamma(f)$ might be empty. The following result describes the possible
behaviours of the iterates.

\begin{result}[\cite{abate1989iteration}, Chapter 2.1]\label{RES:FIX}
  Let $X$ be a taut manifold, and $f \in {\rm Hol}(X,X)$. Then either the sequence
  $\{f^k \}$ of iterates of $f$ is compactly divergent, or there exists a complex 
  submanifold $M$ of $X$ and a holomorphic retraction $\rho:X \to M$ (i.e., $\rho^2 = 
  \rho$) such that every limit point $h \in {\rm Hol}(X,X)$ of $\{f^k\}$ is of the form
  $h = \gamma \circ \rho$, where $\gamma$ is an automorphism of $M$. Moreover,
  \begin{enumerate}
	\item   even $\rho$ is a limit point of the sequence $\{f^k\}$,
	\item   $f|_M$ is an automorphism of $M$.
  \end{enumerate}  
  \label{RES:LIMRET}
\end{result}

\begin{defn}
  With the notation as in Result \ref{RES:LIMRET}, we say that $f$ is 
  \emph{non-recurrent} if the sequence $\{f^k \}$ of iterates of $f$ is compactly 
  divergent. Otherwise, we say that $f$ is \emph{recurrent}, and we call the map
  $\rho$ the \emph{limit retraction}, and the manifold $M$ the \emph{limit manifold}.
\end{defn}

The behaviour of the iterates of a holomorphic self-map of a taut manifold $X$ depends
on whether $f$ has a fixed point or not. The following theorem gives a quantitative description
of the behaviour of the complex
derivative $f'$ at a fixed point of $f$; see \cite[Theorem~2.1.21]{abate1989iteration}
for a proof.

\begin{result}\label{RES:CAR}
  Let $X$ be a taut complex manifold, and let $f \in {\rm Hol}(X,X)$ have some fixed
  point $z_0 \in X$. Then
  \begin{enumerate}
	\item the spectrum of $f'(z_0)$ is contained in $\overline{\disk}$;
	\item $T_{z_0}X$ admits a $f'(z_0)$-invariant splitting $T_{z_0}X = L_N \oplus L_U$
	  such that the spectrum of $f'(z_0)|_{L_N}$ is contained in $\disk$, the spectrum 
	  of $f'(z_0)|_{L_U}$ is contained in $\bdy \disk$ and $f'(z_0)|_{L_U}$ is 
	  diagonalizable;
	\item $L_U$ is the complex tangent space at $z_0$ of the limit manifold of $f$.
  \end{enumerate}
\end{result}

\noindent The space $L_U$ is called the \emph{unitary space} of $f$ at $z_0$.
\smallskip 

The next result gives quite precise information about the set $\Gamma(f)$; see 
\cite[Corollary~2.4.4]{abate1989iteration} for a proof.

\begin{result}\label{RES:GAM}
  Let $X$ be a taut manifold, and let $f \in {\rm Hol}(X,X)$ be recurrent with limit 
  retraction $\rho: X \to M$. Then $\Gamma(f)|_M$ is isomorphic to a compact abelian
  subgroup of ${\rm Aut}(M)$; in particular, it is the closed subgroup generated by $f|_M \in 
  {\rm Aut}(M)$.  
\end{result}

\noindent{We point out that Result~\ref{RES:LIMRET} guarantees that $f|_M$ is an
automorphism of $M$.}
\smallskip

To conclude this section, we reiterate that one of the main results of \cite{opshtein2006dynamics}
plays a central role in the final step in our proof of Theorem~\ref{THM:BAL}. This is 
Result~\ref{RES:OPHI} above. All the terms occurring in that result have now been defined in this
section.
\smallskip

\section{Some essential propositions}\label{SEC:BALESS}

We begin with some material on complex geodesics. In what follows, given a domain $D \subset \C^n$,
we will denote the Kobayashi pseudo-distance by $K_D$, and the infinitesimal Kobayashi
metric by $\kappa_D$. We will denote the Poincar\'{e}
metric and distance on $\disk$ by $\omega$ and $\poin$, respectively.

\smallskip

The notion of complex geodesics was introduced by Vesentini (\cite{ves1981geo})
and is a useful tool in the study of holomorphic mappings.

\begin{defn}
  Let $D \subset \C^n$ be a domain, and let $\phi:\disk 
  \to D$ be a holomorphic map.
  \begin{enumeratei}
   \item Let $a, b \in D$. The map $\phi$ is said to be a \emph{$K_D$-geodesic for $(a,b)$}
   if there exist points $x, y \in \disk$ such that $\phi(x) = a$, $\phi(y) = b$, and
   $\poin(x,y) = K_D(a,b)$.
 \item Let $a \in D$ and $V \in \C^n \setminus \{0\}$. The map $\phi$ is said to be a
   \emph{$\kappa_D$-geodesic for $(a,V)$} if there exists a number
   $\alpha \in \C$ and a point $x \in \disk$ such that $\phi(x) = a$, 
   $V = \alpha\phi^\prime(x)$, and $\kappa_D(a; V) = \omega(x; \alpha)$.
  \end{enumeratei} 
\end{defn}

We need one more notion before we can state an important uniqueness result for certain
geodesics. 

\begin{defn}\label{DEF:HOLEXT}
  Let $D \subset \C^n$ be a bounded domain. We say that a point $a \in \bdy D$ is
  \emph{holomorphically extreme} if there is no non-constant holomorphic mapping $\phi :
  \disk \to \overline{D}$ such that $\phi(0) = a$.
\end{defn}

\begin{example}\label{EX:DSK}
  Every boundary point of a smoothly bounded pseudoconvex domain of finite type is          
  holomorphically extreme. This follows
  easily from the fact that any such point admits a continuous plurisubharmonic barrier. 
  The existence of a continuous plurisubharmonic barrier follows from
  \cite{sibony1987bregular} 
  together with \cite{catlin1984boundary}. Thus every boundary point of $\OM$ in
  Theorem~\ref{THM:BAL} is holomorphically extreme.
\end{example}

The following uniqueness result for $K_D$-geodesics illustrates the importance of the
above definition; see \cite[Proposition 8.3.5]{jarpflug1993invariant} for a proof. 

\begin{result}\label{RES:UNIQ}
  Let $D \subset \C^n$ be a bounded balanced pseudoconvex domain, and let $a \in D, a
  \neq 0$, be
  such that $a/M_D(a) \in \bdy D$ is holomorphically extreme, where $M_D$ is the 
  Minkowski functional of $D$. Then the mapping 
  \[
	\phi_a:\disk \ni \lambda \mapsto \lambda a/M_D(a)
  \]
  is the unique (modulo ${\rm Aut}(\disk)$) $K_D$-geodesic ($\kappa_D$-geodesic) for
  $(0,a)$ (resp., $(0,a/M_D(a))$). 
\end{result}

We now prove a proposition that is a simple consequence of the above uniqueness result. 
This proposition has already been obtained by Vesentini in the more general context of
reflexive Banach spaces. We use his techniques to give a simple proof in our special case.

\begin{prop}\label{PROP:FIX}
  Let $D \subset \C^n$ be a bounded, balanced pseudoconvex domain, all of whose boundary
  points are holomorphically extreme. Let $\rho: D \to D$ be a holomorphic retraction
  such that $\rho(0) = 0$. Then $M := \rho(D) = D \cap V$, where
  \[
	V := \{ v \in \C^n : \rho'(0) v = v \}.
  \]
\end{prop}

\begin{proof}
  If $\rho\equiv 0$, then there is nothing to prove. So assume that $\rho$ is non-constant.
  From Result \ref{RES:FIX}, it follows that $M$ is a connected complex submanifold of 
  $D$ whose complex tangent space at $0$ is $V$. Let $v \in V \cap D$, $v \neq 0$. From
  Result~\ref{RES:UNIQ} and our assumption on $\bdy D$, it follows that the mapping 
  \[
	\phi:\disk \ni \lambda \longmapsto \lambda v/M_D(v) \in D
  \]
  is the unique (modulo ${\rm Aut}(\disk)$) $\kappa_D$-geodesic for
  $(0, v/M_D(v))$. Note that
  \[
	\kappa_D(\rho \circ \phi(0); (\rho \circ \phi)'(0)) =  
	\kappa_D(0; v/M_D(v)) = \kappa_D(\phi(0); v/M_D(v)),	
  \]
  whence $\rho\circ\phi$ is also a $\kappa_D$-geodesic for $(0, v/M_D(v))$. 
  By uniqueness, it follows that $\rho \circ \phi = \phi\circ \psi$, where
  $\psi \in {\rm Aut}(\disk)$. Substituting $\psi^{-1}(M_D(v))$\,($\in \disk$) into
  this equation, we see that $v$ lies in the image of $\rho$. Hence $V \cap D \subset M$.
  Now, $V$ and $M$ have the same dimension. Since
  $M$ is connected, it follows from the principle of analytic 
  continuation that $M = D \cap V$.  
\end{proof}

\section{Proofs of Proposition~\ref{PROP:LOCUS} and Theorem~\ref{THM:BAL}}\label{SEC:BALPROOF}

If we assume that the map $F: \OM \to \OM$, as stated in Theorem~\ref{THM:BAL} is
branched, then our assumptions on the geometry of $\OM$  stated in the first sentence of
Theorem~\ref{THM:BAL} give us a structural result for 
the branch locus $V_F$ of $F$. We begin with the proof of this result --- i.e.,
Proposition~\ref{PROP:LOCUS}. 
We mention here that some of the arguments used in the proof
were inspired by the arguments used in \cite{cof2007trans}.
\smallskip

A comment about notation: in what follows, $\dim_H(S)$ will denote the Hausdorff dimension
of the set $S \subset \C^n$.

\begin{proof}[The proof of Proposition~\ref{PROP:LOCUS}]
  Let $X_1, \ldots,X_m$ be the distinct irreducible
  components of the variety $V_F$. By Lemma \ref{LEM:EXTBELL}, $F$ extends 
  holomorphically to a neighbourhood $N$ of $\overline{\OM}$. For the moment, let 
  $\widetilde{F}$ denote this extension.

  \begin{lemma}\label{LEM:DIM}
	Let $X$ be an arbitrary irreducible component of the variety $V_F$ (viewed as a
	subvariety of $\OM$). Let $E := \overline{X} \setminus X$. There exists a non-empty
	open (relative to $E$) subset $\omega_X \subset E$ such that, for each $p
	\in \omega_X$, there exists a connected neighbourhood $N_p \ni p$ satisfying
	\[
	  \dim_H(E \cap N_p) \geq 2n - 3.
	\]
  \end{lemma}

  \begin{proof}
	Let $\widetilde{F}$ be as introduced just prior to the lemma. Note that $X$ lies in an
	irreducible component $\widetilde{X}$ of $V_{\widetilde{F}}$. Let $S$ denote the
	subvariety of singular points of $\widetilde{X}$. Applying the maximum modulus
	principle on the irreducible variety $X$ to the functions
	\[
	  \left.\frac{\partial \cjac(F)}{\partial z_j}\right|_X, \; \; j =1,2,\dots,n
	\]
	(which vanish simultaneously in $X$ precisely on $S \cap X$), we see that $E \cap S
	\varsubsetneq E$. Let $\omega_X := E \setminus S$. Since non-singularity is an open
	condition on $\widetilde{X}$, $\omega_X$ is open relative to $E$. Pick a point $p \in
	\omega_X$. By definition, $\exists r_p > 0$ such that $\overline{B}(p,r_p) \cap S =
	\emptyset$ and $\widetilde{X} \cap B(p,r_p)$ is a complex submanifold of $B(p,r_p)$. 
    \smallskip
	
	\noindent {\bf Claim:} \emph{For each $r \in (0,r_p)$, $(\widetilde{X} \setminus 
	\overline{\OM}) \cap B(p,r) \neq \emptyset$}.
	
	\noindent Assume this is false. Then, $\exists r \in (0,r_p)$ such that $\widetilde{X}
	\cap B(p,r) \subset \overline{\OM}$. This implies that there exists a non-constant 
	holomorphic map $\psi: \disk \to \C^n$ such that $\psi(\disk) \subset \widetilde{X}
	\cap B(p,r)
	\cap \overline{\OM}$ and $\psi(\disk) \cap \bdy \OM \neq \emptyset$. But this is 
	impossible as every point of $\bdy \OM$ is holomorphically extreme (see Example 
	\ref{EX:DSK}). Hence the claim.
	\smallskip

	Now, let $r_p^* \in (0,r_p)$ be so so small that $B(p,r_p^*) \setminus \bdy \OM$ has 
	exactly
	two connected components (possible as $\bdy \OM$ is an imbedded smooth submanifold),
	\[
	  B(p,r_p^*) \setminus \bdy \OM = C^+ \sqcup C^-.
	\]
	By our above claim, $C^{\pm} \cap \widetilde{X} \neq \emptyset$. Thus $\bdy \OM \cap 
	\widetilde{X}
	\cap B(p,r_p^*)$ disconnects the manifold $\widetilde{X} \cap B(p, r_p^*)$.
	\smallskip

	In what follows, $\dim_I$ will denote the \emph{inductive dimension}. The precise
	definition is rather involved and we refer the reader to 
	\cite[Chapters II and III]{hurwit1941dim}. The fact
	that we need is Corollary 1 to \cite[Theorem IV.4]{hurwit1941dim}: since $\bdy \OM 
	\cap \widetilde{X} \cap B(p,r_p^*)$ disconnects $\widetilde{X} \cap
	B(p,r_p^*)$
	\begin{equation}
	  \dim_I(\bdy \OM \cap \widetilde{X} \cap B(p,r_p^*)) \geq \dim_{\R}(\widetilde{X}
	  \cap B(p,r_p^*)) - 1 = 2n -3.
	  \label{EQ:DIM}
	\end{equation}
	It is well-known that the Hausdorff dimension dominates the inductive dimension; see for
	instance \cite[Chapter VII, \S 4]{hurwit1941dim}. By \eqref{EQ:DIM}, therefore, writing 
	$N_p : =B(p,r_p^*)$,
	\[
	  \dim_H(N_p \cap E) \geq 2n-3.
	\]
	Since the above is true for any $p \in \omega_X$, we are done.
  \end{proof}

  Since the extension $\widetilde{F}$ of $F$ is uniquely determined by the latter, for the
  remainder of the proof we shall not use different symbols for $F:\OM \to \OM$ and
  $\widetilde{F}$.
  \smallskip

  Let $E$ be as in Lemma~\ref{LEM:DIM}. Let us fix a point $z_0 \in E$ for the moment. 
  From
  Result~\ref{RES:TAU}, all the points of $E$ are necessarily weakly pseudoconvex. Note 
  that $\tau(e^{i\theta} z_0) = \tau(z_0)$. From the fact that $\tau$ is upper 
  semi-continuous, it follows that 
  the set $\{w \in \bdy \OM : \tau(w) < \tau(F(z_0)) + 1\}$ is open in $\bdy \OM$, and
  consequently so is its inverse image under $F$, $\{z \in \bdy \OM : \tau(F(z)) <
  \tau(F(z_{0})) + 1 \}$. The latter set obviously contains $z_0$. This implies that for
  $\theta$ close to $0$, we must have $\tau(F(e^{i\theta}z_0)) \leq \tau(F(z_0)) <
  \tau(z_0) = \tau(e^{i\theta}z_0)$ which, by Result~\ref{RES:TAU}, implies that for
  $\theta$ close to $0$, we have $e^{i\theta}z_0 \in \overline{V_F}$. Restricting 
  $\cjac(F)$ to the
  set $\disk \bcdot z_0$, and observing that the boundary-values of this restriction
  vanish on an arc of $\bdy\disk$, we see that
  $\cjac(F)$ must vanish on the set $\disk \bcdot z_0$. As $z_0\in E$ was arbitrary,
  we get that for each 
  $z \in E,\ \disk \bcdot z \subset X_i$, for some $i$. Let us define
  \[
	E_i := \{w \in E: \disk \bcdot w \subset X_i\}, \ i =1,\ldots,m.
  \]
  Since, by Lemma~\ref{LEM:DIM}, $E$ is of Hausdorff dimension at least $2n - 3$, there is
  an $i_0, 1 \leq i_0 \leq m$, such that $\dim_H(E_{i_0}) \geq 2n - 3$. Let us call this
  set $E'$. Then: 
  \[ 
	\bigcup_{ z \in E'}\disk \bcdot z \subset X_{i_0}.
  \]
  As $X$ and $X_{i_0}$ are irreducible varieties the intersection of whose closures is of
  Hausdorff dimension at least $2n - 3$, it must be that $X = X_{i_0}$. This follows from
  the stratification of $X \cap X_{i_0}$ \cite[Chapter~1, \S\,5.5]{chirka1989complex}, the
  properties of the Hausdorff measure and 
  \cite[Corollary~1 of Chapter~1, \S\,5.3]{chirka1989complex}.
  \smallskip

  We have proved that 
  \begin{equation}\label{EQ:DSKX}
	\bigcup_{ z \in E'}\disk \bcdot z \subset X.
  \end{equation}
  Now fix $\lambda \in \disk$, and consider the holomorphic function $h_{\lambda}(z) := 
  \cjac(F)(\lambda z)$ defined on $X$. From what we have shown, \eqref{EQ:DSKX}
  in particular, $h_{\lambda}$ vanishes on
  a subset of Hausdorff dimension at least $2n -3$ of the irreducible variety $X$. Hence 
  $h_{\lambda}$ must vanish identically on $X$, and this is true for each $\lambda \in
  \disk$. Thus, we have shown that, given $z \in X$, $(\disk \bcdot z) \cap \OM \subset
  V_F$. Since $V_F$ comprises finitely many irreducible components, by a similar
  argument as in the previous paragraph (with the role of $E$ now taken by $X$),
  we actually have $(\disk \bcdot z) \cap \OM \subset X$. 
  By analytic continuation, it follows that if $z \in X$, then $(\C \bcdot z) \cap \OM
  \subset X$. 
\end{proof}

We now have all the tools needed to prove our main theorem. Owing to the fact that the domain
$\OM$ admits a smooth defining function that is plurisubharmonic in $\OM$, it is 
\emph{a fortiori} pseudoconvex. We shall use this fact without explicit mention in our proof.

\begin{proof}[The proof of Theorem~\ref{THM:BAL}]
  Observe that $\OM$ is contractible. Thus, to prove that
  $F$ is an automorphism it suffices, in view of the monodromy theorem, to
  establish that $F$ is unbranched. So we will 
  assume that $F$ is branched and reach a contradiction. We will not, hereafter, remark
  upon the well-definedness of quantities such as $\cjac(F)(p)$ for $p\in \bdy\OM$.
  In view of Lemma~\ref{LEM:EXTBELL}, these are indeed well-defined.

  \vspace{0.1in}

  \noindent{\bf Step 1.} \emph{Proving that $F^{-1}\{0\} = \{0\}$.}
  If not, there is point $0 \neq z_0 \in \OM$ such that $F(z_0) = 0$. From 
  Lemma~\ref{LEM:VAR}, we have distinct irreducible subvarieties $L_i \subset V_{F^i}$ such 
  that $F|_{L_{i+1}}: L_{i+1} \to L_i$ is a proper holomorphic mapping. By Proposition~\ref{PROP:LOCUS},
  each $L_i$ contains $0$, whence we can select each $L_i$ as described in Lemma~\ref{LEM:VAR} 
  and ensure that $z_0 \in L_i \ \forall i\geq 2$. Again from Proposition~\ref{PROP:LOCUS}, it follows
  that $\Lambda := (\C \bcdot z_0) \cap \OM \subset L_i \ \forall i\geq 2$. This means that 
  the sets $F^k(\Lambda) \subset L_2, \forall k \in \Z_+$. Let $q \in (\C \bcdot z_0) \cap 
  \bdy \OM = \overline{\Lambda} \setminus \Lambda$. It is elementary that
  \[
   V_{F^n} = \bigcup_{k=0}^{n-1}(F^k)^{-1}(V_F), \quad n\in \Z_+,
  \]
  with the understanding that $F^0 = {\sf id}_{\Omega}$. Thus --- we refer to the
  recipe for the $L_i$'s in Lemma \ref{LEM:VAR} ---   $L_2 \subset V_{F^{2k}} \forall k \in \Z_+$.
  Note that $q\in \overline{L_2}\setminus L_2$. At this stage, we are precisely in the situation
  prior to \eqref{EQ:TYPE_DES}
  in the proof of Lemma~\ref{LEM:VAR}, except that $q$ belongs to the (prolongation of)
  the branch locus of $F^2$. Therefore, it follows as in the
  proof of Lemma~\ref{LEM:VAR} (taking $p=2$ in the relevant argument), that   
  \[
	0 \leq \cdots < \tau(F^{2k}(q)) < \cdots < \tau(F^4(q)) < \tau(F^2(q)) < \tau(q),
  \]
  which contradicts the conclusion of Result~\ref{RES:FIN}. Our claim follows.
 \smallskip

  We should point out (although we shall not make use of it below) that from a theorem of Bell
  \cite[Theorem~1]{bell1982proper}, it follows that $F$ is a polynomial mapping. 
  
  \vspace{0.1in}

  \noindent{\bf Step 2.} \emph{Analyzing the limit manifold of $F$.}
  As $0$ is a fixed point of $F$, it follows that $F$ is recurrent. Let $\rho : \OM \to M$
  be the limit retraction. As $\OM$ is pseudoconvex and of finite type, by Example~\ref{EX:DSK}
  every point in $\bdy \OM$ is holomorphically extreme. Consequently, from Proposition
  \ref{PROP:FIX} and Result \ref{RES:CAR} it follows that $M = L_U \cap \OM$, where $L_U$
  is the unitary space of $F$ at $0$. Recall that $F'(0)|_{L_U}$ is diagonalizable; see
  Result~\ref{RES:CAR}. So without loss of generality, (replacing $\OM$ by a suitable
  linear image and conjugating $F$ by a suitable linear operator, if needed) we may 
  assume that $L_U = \C^m \times \{0_{\C^{n-m}}\}$, and that $F'(0)|_{L_U}$ is given by 
  \begin{equation}\label{EQ:LIN}
	(z_1,z_2,\ldots,z_m,0,\ldots,0) \mapsto (e^{i\theta_1}z_1,e^{i\theta_2}z_2,\ldots,
	e^{i\theta_m}z_m,0,\ldots,0). 
  \end{equation}
  By Cartan's uniqueness theorem, it also follows that $F|_M \equiv F'(0)|_M$.

  \vspace{0.1in}

  \noindent{\bf Step 3.} \emph{Proving that $\dim M \leq 1$.}
  Suppose $\dim M > 1$. From the previous steps, we have that $0 \in V_F \cap M$. Therefore
  the set $\{z \in M : \cjac(F)(z) = 0\}$ is a non-empty analytic subvariety of $M$. As 
  $\dim M > 1$, it follows that there exists a point $p \in \overline{M} \cap \bdy \OM$
  such that $\cjac(F)(p) = 0$. 
  \smallskip

  From Result~\ref{RES:GAM}, and the fact that the maps $F^k|_M, k \in \Z$,
  are all of the special form \eqref{EQ:LIN},
  it follows that there is a strictly increasing sequence $\{n_k\} \in \N$ such that
  $F^{n_k}|_{\overline{M}} \to (F|_{\overline{M}})^{-1}$ uniformly on $\overline{M}$.
   Let 
	$p_k := (F|_{\overline{M}})^{-n_k}(p)$. By Result~\ref{RES:TAU}, it follows that
  $\tau(p_{k+1}) \geq \tau(p_k) \geq \tau(p)$. But, $p_k \to F(p)$, and as $p \in 
  \overline{V}_F, \tau(F(p)) < \tau(p)$, which means that
  \[
	\limsup_{k \to \infty} \tau(p_k) \geq \tau(p) > \tau(F(p)),
  \]
  which contradicts the fact that $\tau$ is upper semi-continuous. This proves that
  $\dim(M) \leq 1$.

  \vspace{0.1in}

  The conclusion of Step 3 is in conflict with the conclusion of Opshtein's theorem, i.e.,
  Result~\ref{RES:OPHI}. Therefore, $F$ is unbranched. As $\OM$ is simply connected, it
  follows from the monodromy theorem (just continue the germ of a local inverse of $F$
  at $0$ along paths in $\OM$ from $0$ to $z\in \OM$: the germ obtained at $z$ will be a
  local inverse by the permanence of relations) that $F$ is an automorphism. 
\end{proof}

\noindent {\bf Acknowledgements.} I would like to thank my advisor Gautam Bharali
for his support during the course of this work, and for suggesting several
useful ideas. The idea behind Lemma~\ref{LEM:DIM} was given by him. I 
would also like to thank my colleague and friend G.P. Balakumar for some informative 
discussions related to this work. I am truly grateful to the anonymous referee of this
article for the many suggestions for improving the exposition.

\end{document}